\documentclass[10pt,a4paper]{amsart}
\usepackage{amssymb,amsmath,amsfonts}

\headsep=1.2500cm \numberwithin{equation}{section}
\newtheorem{Theorem}{Theorem}[section]

\newtheorem{Lemma}[Theorem]{Lemma}
\theoremstyle{remark}
\newtheorem{Definition}[Theorem]{Definition}
\newtheorem{Example}[Theorem]{Example}

\begin{document}

\title{On the existance of an ultra central approximate identity for certain semigroup algebras}

\author[A. Sahami]{A. Sahami}

\email{amir.sahami@aut.ac.ir}

\author[I. Almasi]{I. Almasi}

\email{i.almasi@aut.ac.ir}

\address{Department of Mathematics
Faculty of Basic Sciences Ilam University P.O. Box 69315-516 Ilam,
Iran.}



\maketitle

\begin{abstract}
In this paper we characterize  the existance of an ultra central approximate identity for
$\ell^{1}(S)$, where $S$ is a uniformly locally finite inverse
semigroup. As an application, for   the  Brandt semigroup $S=M^{0}(G,I)$ over a
non-empty set $I$, we show that $\ell^{1}(S)$ has an ultra central approximate identity  if
and only if  $I$ is finite.
\end{abstract}
\section{Introduction and Preliminaries}
\begin{Definition}
Let $A$ be a Banach algebra. We say that $A$ has an ultra central approximate identity if there exists a net $(e_{\alpha})$ in $A^{**}$ such that $ae_{\alpha}=e_{\alpha}a$ and $e_{\alpha}a\rightarrow a,$ for every $a\in A.$
\end{Definition}
It is easy to see that every Banach algebra $A$ with central approximate identity has an ultra central approximate identity. 
We will see that every Banach algebra with a bounded approximate has an ultra central approximate identity but the converse is not always true. 
We will prove that every  Banach algebra with a bounded approximate identity or central approximate identity has an ultra central approximate  identity. Thus the class of Banach algebras which has an ultra central approximate identity is abundant. In fact it is well-known that for a locally compact group $G$, $L^{1}(G)$   has a bounded approximate identity. Also using the main result of \cite{kot} we know that  $S^{1}(G)$(the Segal algebra with respect to a locally compact group $G$)  has a central approximate identity if and only if $G$ is a $SIN$ group. 

 Recently 
Ramsden in \cite[Proposition 2.9]{rams} has been showed that if a semigroup algebra $\ell^{1}(S)$ has a bounded approximate identity, then the set of idempotent elements of $S$, say $E(S)$, is finite, provided that $S$ is an uniformly locally finite semigroup. In fact, he gave a relation between the topological notion of bounded approximate identity and  the algebric notion of idempotent set. So the following question raised

"{\it {\bf What will happen if $\ell^{1}(S)$ has an ultra central approximate identity}}?" 

Since the structure of  the uniformly locally finite inverse semigroup algebra  is related to some group algebras, we answer this question for the semigroup algebras associated to an  uniformly locally finite  inverse semigroups.
In fact (motivated by \cite[Example 4.1(iii)]{sah})we show that  $M_{\Lambda}(\mathbb{C})$ (the Banach algebra of $\Lambda\times
\Lambda $-matrices over $\mathbb{C}$, with finite $\ell^{1}$-norm and
matrix multiplication) has an ultra central approximate identity if and only if $\Lambda$ is finite. Using this tool
 we characterize  the existance of  an ultra central approximate identity for the  semigroup algebra $\ell^{1}(S)$, provided that $S$ is an uniformly locally finite semigroup. As an application, we show that $\ell^{1}(S)$ has an ultra central approximate identity  if
and only if  $I$ is finite, where  $S=M^{0}(G,I)$ is the Brandt semigroup over a
non-empty set $I$.

First we present  some standard notations and definitions that we
shall need in this paper. Let $A$ be a Banach algebra. If $X$ is a
Banach $A$-bimodule, then  $X^{*}$ is also a Banach $A$-bimodule via
the following actions
$$(a\cdot f)(x)=f(x\cdot a) ,\hspace{.25cm}(f\cdot a)(x)=f(a\cdot x ) \hspace{.5cm}(a\in A,x\in X,f\in X^{*}). $$
Let $A$ and $B$ be   Banach algebras. The projective tensor product
 $A\otimes_{p}B$ with the following multiplication is a Banach algebra
$$(a_{1}\otimes b_{1})(a_{2}\otimes b_{2})=a_{1}a_{2}\otimes b_{1}b_{2}\quad (a_{1},a_{2}\in A, b_{1}b_{2}\in B).$$
 The product morphism
$\pi_{A}:A\otimes_{p}A\rightarrow A$ is  specified by
$\pi_{A}(a\otimes b)=ab$ for every $a,b\in A$.

Let $A$ be a Banach algebra and  let $\Lambda$ be a non-empty set.
We denote $\varepsilon_{i,j}$ for a matrix belongs to
$\mathbb{M}_{\Lambda}(A)$ which $(i,j)$-entry is 1 and 0 elsewhere.
The map $\theta:\mathbb{M}_{\Lambda}(A)\rightarrow A\otimes_{p}
\mathbb{M}_{\Lambda}(\mathbb{C})$ defined by
$\theta((a_{i,j}))=\sum_{i,j}a_{i,j}\otimes \varepsilon_{i,j}$ is an
isometric algebra isomorphism.

We present  some notions of semigroup theory, for the further
background see \cite{how}. Let $S$ be a semigroup and let $E(S)$ be
the set of its idempotents. There exists a  partial order on $E(S)$ which is defined by
$$s\leq t\Longleftrightarrow s=st=ts\quad (s,t\in E(S)).$$ A  semigroup $S$ is called inverse semigroup, if for every $s\in S$
there exists $s^{*}\in S$ such that $ss^{*}s=s^{*}$ and
$s^{*}ss^{*}=s$. If $S$ is an inverse semigroup,
then there exists a partial order on $S$ which  coincides with the
partial order on $E(S)$. Indeed
$$s\leq t\Longleftrightarrow s=ss^{*}t\quad (s,t\in
S).$$ For every  $x\in S$, we denote $(x]=\{y\in S|\,y\leq x\}$. $S$
is called locally finite (uniformly locally finite) if for each
$x\in S$, $|(x]|<\infty\,\,(\sup\{|(x]|\,:\,x\in S\}<\infty)$,
respectively.

Suppose that $S$ is an inverse semigroup. Then  the maximal subgroup
of $S$ at $p\in E(S)$ is denoted by $G_{p}=\{s\in
S|ss^{*}=s^{*}s=p\}$.

 Let $S$ be an inverse semigroup. There exists an equivalence relation $\mathfrak{D}$ on $S$
such that $s\mathfrak{D}t$ if and only if there exists $x\in S$ such
that $ss^{*}=xx^{*}$ and $t^{*}t=x^{*}x$. We denote
$\{\mathfrak{D}_{\lambda}:\lambda\in \Lambda\}$ for the collection
of $\mathfrak{D}$-classes and $E(\mathfrak{D}_{\lambda})=E(S)\cap
\mathfrak{D}_{\lambda}.$

\section{Main results}
\begin{Theorem}\label{matris}
Let $\Lambda$ be any non-empty set. Then
$\mathbb{M}_{\Lambda}(\mathbb{C})$  has an ultra central approximate identity if
and only if $\Lambda$ is finite.
\end{Theorem}
\begin{proof}
 Suppose that
$A=\mathbb{M}_{\Lambda}(\mathbb{C})$  has an ultra central approximate identity.
It follows that there exists a net $(e_{\alpha})$ in
$A^{**}$ such that $a\cdot e_{\alpha}=e_{\alpha}\cdot
a$ and $e_{\alpha}a\rightarrow a$ for each $a\in A.$
Suppose that  $a$ is  any non-zero element of $A$. Using the  Hahn-Banach theorem,  we have  a bounded linear functional $\Upsilon$ in $A^{*}$ such that
$\Upsilon(a)\neq 0.$ Since the  convergence of a net  with respect to the norm topology implies the convergence with respect to the $w^{*}-$topology,  we have 
$e_{\alpha}a(\Upsilon)\rightarrow a(\Upsilon)$. 
Thus 
$e_{\alpha}(a\cdot \Upsilon)\rightarrow
\Upsilon(a)\neq 0$. 
Then without loss of generality, we may suppose that 
$e_{\alpha}(a\cdot \Upsilon)\neq 0$ for each
$\alpha$. By Alaghlou's theorem, there exists a  bounded net
$(x^{\beta}_{\alpha})$ with the bound $||e_{\alpha}||$ in
$A$ such that
$x^{\beta}_{\alpha}\xrightarrow{w^{*}}e_{\alpha}$ .
On the other hand we have 
$a\cdot x^{\beta}_{\alpha}\xrightarrow{w^{*}}a\cdot e_{\alpha}$ and
$x^{\beta}_{\alpha}\cdot a\xrightarrow{w^{*}}e_{\alpha}\cdot a$ for
each $a\in A.$
Therefore 
$a\cdot
x^{\beta}_{\alpha}-x^{\beta}_{\alpha}\cdot a\xrightarrow{w^{*}}0$
(and also since $(x^{\beta}_{\alpha})$ is a net in $A$
we have $a\cdot x^{\beta}_{\alpha}-x^{\beta}_{\alpha}\cdot
a\xrightarrow{w}0$).
Fix $\alpha$ and set $y_{\beta}=x^{\beta}_{\alpha}$.  It is easy to see that 
$(y_{\beta})_{\beta}$
is a bounded net in $A$ such that  $a y_{\beta}-y_{\beta}
a\xrightarrow{w}0$ and
$y_{\beta}\xrightarrow{w^{*}}e_{\alpha}$ for each
$a\in A$.
Let  $y_{\beta}=(y^{i,j}_{\beta})$, where
$y^{i,j}_{\beta}\in \mathbb{C}$,  for every $i,j\in\Lambda$.
It is well-known that  the product
of the weak topology on $\mathbb{C}$ coincides with the weak
topology on $A$ \cite[Theorem 4.3]{schae}. Then  for each 
$i_{0}\in\Lambda$ we have
$\varepsilon_{i_{0},j}y_{\beta}-y_{\beta}\varepsilon_{i_{0},j}\xrightarrow{w}0$.
It gives that 
$y^{j,j}_{\beta}-y^{i_{0},i_{0}}_{\beta}\xrightarrow{w}0$ and
$y^{i,j}_{\beta}\xrightarrow{w}0$, whenever $i\neq j$.
On the other hand boundedness of  the net 
$(y_{\beta})$ implies that  $(y^{i_{0},i_{0}}_{\beta})$ is a
bounded net in $\mathbb{C}$.
So the net  $(y^{i_{0},i_{0}}_{\beta})$ in $\mathbb{C}$ has a
$w^{*}$-convergence subnet. Since $\mathbb{C}$ is a Hilbert space we can assume that the subnet has a $w$-convergence subnet, with $w$-limit point $l$, say
$(y^{i_{0},i_{0}}_{\beta_{k}})$. Again since  the net $(y^{i_{0},i_{0}}_{\beta_{k}})$ belongs to $\mathbb{C}$ we may assume that the convergence happens in the norm topology,  so
 $y^{i_{0},i_{0}}_{\beta_{k}}\xrightarrow{|\cdot|} l$.
Now the fact 
$y^{j,j}_{\beta}-y^{i_{0},i_{0}}_{\beta}\xrightarrow{w}0$, implies that 
$y^{j,j}_{\beta}-y^{i_{0},i_{0}}_{\beta}\xrightarrow{|.|}0$.
Thus by 
$y^{j,j}_{\beta_{k}}-y^{i_{0},i_{0}}_{\beta_{k}}\xrightarrow{|.|}0$, we have  $y^{j,j}_{\beta_{k}}\xrightarrow{|.|}l$ for each $j\in\Lambda.$
We claim that $l\neq 0$.
Suppose  that $l=0$. So by 
\cite[Theorem 4.3]{schae} we have $y_{\beta}\xrightarrow{w}0$. It implies that 
$(a\cdot\Upsilon)(y_{\beta})\rightarrow 0$. 
Since 
$(a\cdot\Upsilon)(y_{\beta})=y_{\beta}(a\cdot\Upsilon)\rightarrow e_{\alpha}(a\cdot
\Upsilon)\neq 0$,  we have a contradiction.   Thus $l\neq 0$.
On the other hand 
$y^{j,j}_{\beta_{k}}-y^{i_{0},i_{0}}_{\beta_{k}}\xrightarrow{w}0$
and $y^{i,j}_{\beta_{k}}\xrightarrow{w}0$ by \cite[Theorem 4.3]{schae}. 
It follows that  $y_{\beta_{k}}\xrightarrow{w}y_{0}$, where
$y_{0}$ is denoted for a matrix with $l$ in the  diagonal position  and $0$
elsewhere.
Then  $y_{0}\in
\overline{\hbox{Conv}(y_{\beta})}^{w}=\overline{\hbox{Conv}(y_{\beta})}^{||.||}$. It deduces   $y_{0}\in A$.  Therefore
 $\infty =\sum_{j\in \Lambda}|l|= \sum_{j\in
\Lambda}|y^{j,j}_{0}|=||y_{0}||<\infty,$ provided that
$\Lambda$ is
infinite which is a contradiction. So $\Lambda$ must be finite.

Conversely, suppose that  $\Lambda$ be finite. it is easy to see that  $\mathbb{M}_{\Lambda}(\mathbb{C})$ has an identity, say $e$. Since two maps $a\mapsto ae$ and $a\mapsto ea$ on $\mathbb{M}_{\Lambda}(\mathbb{C})^{**}$ are  $w^{*}$-continuous, we have $e$ as an identity for $\mathbb{M}_{\Lambda}(\mathbb{C})^{**}$. Thus $\mathbb{M}_{\Lambda}(\mathbb{C})$  has an ultra central approximate identity.
\end{proof}
\begin{Lemma}\label{amen}
Let $A$ be an amenable Banach algebra. Then $A$ has an ultra central approximate identity. 
\end{Lemma}
\begin{proof}
Since $A$ is amenable, there exists an element $m\in (A\otimes_{p}A)^{**}$ such that $a\cdot m=m\cdot a$  for each $a,b\in A$,  see \cite{run}. It is easy to see that $\pi^{**}_{A}(m)\in A^{**}$,  $a\pi^{**}_{A}(m)=\pi^{**}_{A}(a\cdot m)=\pi^{**}_{A}(m\cdot a)=\pi^{**}_{A}(m)a$ and $\pi^{**}_{A}(m)a=a$ for every $a\in A.$ So $A$ has an ultra central approximate identity.
\end{proof}
We recall that a Banach algebra $A$ is called pseudo-contractibe if there exists a net $(m_{\alpha})$ in $A\otimes_{p}A$ such that 
$a\cdot m_{\alpha}=m_{\alpha}\cdot a$ and $\pi_{A}(m_{\alpha})a\rightarrow a $ for every $a\in A$, see \cite{ghah pse}.
\begin{Lemma}
Let $A$ be a pseudo-contractible Banach algebra. Then $A$ has an ultra central approximate identity.
\end{Lemma}
\begin{proof}
Since $A$ is pseudo-contractible, there exists a net $(m_{\alpha})$ in $A\otimes_{p}A$ such that 
$a\cdot m_{\alpha}=m_{\alpha}\cdot a$ and $\pi_{a}(m_{\alpha})a\rightarrow a $ for every $a\in A$. Set $e_{\alpha}=\pi_{A}(m_{\alpha})$. It is easy to see that $$ae_{\alpha}=a\pi_{A}(m_{\alpha})=\pi_{A}(a\cdot m_{\alpha})=\pi_{A}(m_{\alpha}\cdot a)=\pi_{A}(m_{\alpha})a=e_{\alpha}a$$ and $e_{\alpha}a=\pi_{A}(m_{\alpha})a\rightarrow a$ for every $a\in A.$ Since $A$ can be embedded in $A^{**}$, $(e_{\alpha})$ becomes an ultra  central approximate identity for $A$.
\end{proof}
Clearly one can show  that every Banach algebra with a central approximate identity has an ultra central approximate identity.
Also similar to the proof of Lemma \ref{amen}, we can show that every Banach algebra with a bounded approximate identity has an ultra central approximate identity.
\begin{Example}
Let $S=\mathbb{N}$. With {\it min} as its multiplication, $S$ becomes a commutative semigroup. Let $w:S\rightarrow [1,\infty) $ be any function. It is easy to show that  $w(st)\leq w(s)w(t)$ for each $s,t\in S.$ So $w$ is a weight on $S.$ Set $A=\ell^{1}(S,w)$, the weighted semigroup algebra with respect to $S.$ Suppose that $w(n)=e^{n}$ for each $n\in S.$ Clearly $\lim w(n)=\infty.$ So by \cite[Proposition 3.3.1]{dales} $A$ doesn't have a bounded approximate identity but it has a central approximate identity. Then we have a Banach algebra with an ultra central approximate identity but it  doesnt have bounded approximate identity.

Let $G$ be a locally compact non-$SIN$ group. Then by  the main result of \cite{kot}, we have $L^{1}(G) $ doesnt have central approximate identity. On the other hand it is well-known that every group algebra on a locally compact group $G$ has a bounded approximate identity. Then $L^{1}(G)$ has an ultra central approximate identity but it does't have a central approximate identity.
\end{Example}
\begin{Lemma}\label{tensor}
Let $A$ and $B$ be Banach algebras which $A$ is unital. If $A\otimes_{p}B$ has an ultra central approximate identity, then $B$ has an ultra central approximate identity.
\end{Lemma}
\begin{proof}
Suppose that $A\otimes_{p}B$ has an ultra central approximate identity. Then there exists a net  $(e_{\alpha})$ in $(A\otimes_{p}B)^{**}$ such that $xe_{\alpha}=e_{\alpha}x$ and $e_{\alpha}x\rightarrow x$ for every $x\in A\otimes_{p}B$. Using the following actions one may consider  $A\otimes_{p}B$ as  a Banach $B-$bimodule:
$$b_{1}\cdot(a\otimes b_{2})=a\otimes b_{1}b_{2}, \quad (a\otimes b_{2})\cdot b_{1}=a\otimes b_{2}b_{1}\qquad (a\in A,b_{1},b_{2}\in B).$$
Let $e$ be the identity of $A$. By Hahn-Banach theorem we can find $\phi_{e}\in A^{*}$ such that $\phi_{e}(e)=1.$ Define  $\phi_{e}\otimes id_{B}$ from $A\otimes_{p}B$ into $ B$ by $\phi_{e}\otimes id_{B}(a\otimes b)=\phi_{e}(a)b$ for every $a\in A$ and $b\in B$, where $id_{B}$ is denoted for the identity map on $B$. It is easy to see that $\phi_{e}\otimes id_{B}$ is a bounded linear map.  We claim that $((\phi_{e}\otimes id_{B})^{**}(e_{\alpha})_{\alpha})$ is an ultra central approximate identity for $B.$ To see this consider $$b(\phi_{e}\otimes id_{B})^{**}(x)=(\phi_{e}\otimes id_{B})^{**}(b\cdot x),\quad (\phi_{e}\otimes id_{B})^{**}(x)b=(\phi_{e}\otimes id_{B})^{**}(x\cdot b),\quad (x\in (A\otimes_{p}B)^{**}, b\in B).$$ Then we have 

\begin{equation*}
\begin{split}
b(\phi_{e}\otimes id_{B})^{**}(e_{\alpha})=(\phi_{e}\otimes id_{B})^{**}(b\cdot e_{\alpha})&=(\phi_{e}\otimes id_{B})^{**}((e\otimes b) e_{\alpha})\\
&=(\phi_{e}\otimes id_{B})^{**}(e_{\alpha}(e\otimes b) )\\
&=(\phi_{e}\otimes id_{B})^{**}( e_{\alpha}\cdot b)\\
&=(\phi_{e}\otimes id_{B})^{**}( e_{\alpha})b
\end{split}
\end{equation*}
and 
\begin{equation*}
\begin{split}
(\phi_{e}\otimes id_{B})^{**}( e_{\alpha})b=(\phi_{e}\otimes id_{B})^{**}(e_{\alpha}(e\otimes b) )\rightarrow (\phi_{e}\otimes id_{B})^{**}(e\otimes b)= \phi_{e}\otimes id_{B}(e\otimes b)=b,
\end{split}
\end{equation*}
for each $b\in B$. Thus $B$ has an ultra central approximate identity.
\end{proof}
\begin{Theorem}\label{inverse}
Let $S$ be an inverse semigroup such that $E(S)$ is uniformly
locally finite. Then the following are equivalent:
\begin{enumerate}
\item [(i)] $\ell^{1}(S)$ has an ultra central approximate identity;
\item [(ii)] Each $\mathfrak{D}$-class has finitely many
idempotent elements.
\end{enumerate}
\end{Theorem}
\begin{proof}
Suppose that $\ell^{1}(S)$ has an ultra central approximate identity. Then there esists a net $(e_{\alpha})$ in $\ell^{1}(S)^{**}$ such that $ae_{\alpha}=e_{\alpha}a$ and $e_{\alpha}a\rightarrow a$ for each $a\in \ell^{1}(S).$ Using \cite[Theorem
2.18]{rams} since   $S$
is a uniformly locally finite inverse semigroup, we have 
 $$\ell^{1}(S)\cong
\ell^{1}-\bigoplus\{\mathbb{M}_{E(\mathfrak{D}_{\lambda})}(\ell^{1}(G_{p_{\lambda}}))\},$$
where $\{\mathfrak{D}_{\lambda}:\lambda\in \Lambda\}$ is a
$\mathfrak{D}$-class and $G_{p_{\lambda}}$ is a maximal subgroup at
$p_{\lambda}$. We claim that $\mathbb{M}_{E(\mathfrak{D}_{\lambda})}(\ell^{1}(G_{p_{\lambda}}))$ has an ultra central approximate identity. To see this let $P_{\lambda}$ be the projection map from $\ell^{1}(S)$ onto $\mathbb{M}_{E(\mathfrak{D}_{\lambda})}(\ell^{1}(G_{p_{\lambda}}))$. It is easy to see that $$aP^{**}_{\lambda}(e_{\alpha})=P^{**}_{\lambda}(ae_{\alpha})=P^{**}_{\lambda}(e_{\alpha}a)=P^{**}_{\lambda}(e_{\alpha})a$$
and
$$P^{**}_{\lambda}(e_{\alpha})a=P^{**}_{\lambda}(e_{\alpha}a)\rightarrow P^{**}_{\lambda}(a)=a,$$
for every $a\in \mathbb{M}_{E(\mathfrak{D}_{\lambda})}(\ell^{1}(G_{p_{\lambda}}))$. Then $\mathbb{M}_{E(\mathfrak{D}_{\lambda})}(\ell^{1}(G_{p_{\lambda}}))$ has an ultra central approximate identity.  On the other hand we know that 
$\mathbb{M}_{E(\mathfrak{D}_{\lambda})}(\ell^{1}(G_{p_{\lambda}}))\cong
\ell^{1}(G_{p_{\lambda}})
\otimes_{p}\mathbb{M}_{E(\mathfrak{D}_{\lambda})}(\mathbb{C})$. Since $\ell^{1}(G_{p_{\lambda}})$ is a unital Banach algebra, by Lemma \ref{tensor},  we have $\mathbb{M}_{E(\mathfrak{D}_{\lambda})}(\mathbb{C})$ has an ultra central approximate identity. Now applying Lemma \ref{matris} implies that $E(\mathfrak{D}_{\lambda})$ is finite.

Conversely, suppose that   $E(\mathfrak{D}_{\lambda})$ is finite. Since each $\ell^{1}(G_{p_{\lambda}})$ is unital, each $\mathbb{M}_{E(\mathfrak{D}_{\lambda})}(\ell^{1}(G_{p_{\lambda}}))\cong
\ell^{1}(G_{p_{\lambda}})
\otimes_{p}\mathbb{M}_{E(\mathfrak{D}_{\lambda})}(\mathbb{C})$ is a unital Banach algebra. So one can easily see that $\ell^{1}(S)\cong
\ell^{1}-\bigoplus\{\mathbb{M}_{E(\mathfrak{D}_{\lambda})}(\ell^{1}(G_{p_{\lambda}}))\},$  has a central approximate identity. Therefore $\ell^{1}(S)$ has an ultra central approximate identity.
\end{proof}
For a locally compact group $G$ and non-empty set $I$, set
$$M^{0}(G,I)=\{(g)_{i,j}:g\in G,i,j\in I\}\cup \{0\},$$
where $(g)_{i,j}$ denotes the $I\times I$ matrix with  $g$ in
$(i,j)$-position and zero elsewhere. With the following
multiplication $M^{0}(G,I)$ becomes a semigroup
\begin{eqnarray*}�
&\textit{�$(g)_{i,j}\ast
(h)_{k,l}=$}\begin{cases}(gh)_{il}\,\,\,\,\,\,\,\,\,\,\,\,\,\,\,\,\,\,
j=k\cr � 0\,\,\,\,\,\,\,\,\,\,\,\,\,\,\,\,\,\,     \qquad j\neq k,
�\end{cases}\\�
\end{eqnarray*}�
It is well known that $M^{0}(G,I)$ is an inverse semigroup with
$(g)^{*}_{i,j}=(g^{-1})_{j,i}$. This semigroup is called Brandt
semigroup over $G$ with the  index set $I.$

\begin{Theorem}\label{brandt}
Let $S=M^{0}(G,I)$ be a Brandt semigroup over  a  group $G$ with index set $I.$  Then the following
are equivalent:
\begin{enumerate}
\item [(i)] $\ell^{1}(S)$ has an ultra central approximate identity;
\item [(ii)]  $I$ is finite;
\end{enumerate}
\end{Theorem}
\begin{proof}(i)$\Rightarrow$(ii) Suppose that $\ell^{1}(S)$ has an ultra central approximate identity. Using  \cite[Remark, p 315]{dun-nam}, we know that $\ell^{1}(S)$
is isometrically isomorphic with
$[M_{I}(\mathbb{C})\otimes_{p}\ell^{1}(G)]\oplus_{1}\mathbb{C}$. By similar argument as in the proof of Theorem \ref{inverse}(if part) we can see that $M_{I}(\mathbb{C})\otimes_{p}\ell^{1}(G)$ has an ultra central approximate identity. Since $\ell^{1}(G)$ is unital, Lemma \ref{tensor} follows that   $M_{I}(\mathbb{C})$  has an ultra approximate identity. Applying Lemma \ref{matris}, implies that $I$ is finite.

(ii)$\Rightarrow$(i) Since $I$ is finite, $[M_{I}(\mathbb{C})\otimes_{p}\ell^{1}(G)]\oplus_{1}\mathbb{C}$ has an identity. So $\ell^{1}(S)$ is unital. Clearly  $\ell^{1}(S)$ has an ultra central approximate identity.

\end{proof}
\begin{small}

\end{small}

\end{document}